\documentclass[12pt,reqno]{amsart}
\usepackage{amssymb}
\usepackage[usenames,dvipsnames]{color}
\usepackage{euscript}
\usepackage{multicol}
\usepackage{amssymb}
\usepackage{amsmath}
\usepackage{times} 

\newcommand{\Q}{\mathcal{Q}}
\newcommand{\eps}{\varepsilon}

\newcommand{\N}{\mathbb{N}}

\newcommand{\R}{\mathbb{R}}
\newcommand{\sph}{\mathcal{L}}

\newtheorem{theorem}{Theorem}[section]
\newtheorem{corollary}[theorem]{Corollary}
\newtheorem{lemma}[theorem]{Lemma}
\newtheorem{proposition}[theorem]{Proposition}

\theoremstyle{definition}

\newtheorem{remark}[theorem]{Remark}
\newtheorem{assumption}[theorem]{Assumption}

\numberwithin{equation}{section}

\begin{document}

\title[Hardy inequalities for Robin Laplacians]
{Hardy inequalities for Robin Laplacians}

\author {Hynek Kova\v{r}\'{\i}k}

\address {Hynek Kova\v{r}\'{\i}k, Dipartimento di Matematica \\ Politecnico di Torino \\Corso Duca degli Abruzzi 24, Torino10129, Italy}

\email {hynek.kovarik@polito.it}

\author {Ari Laptev}

\address {Ari Laptev, Imperial College London\\ Huxley Building, 180 Queen's Gate \\
London SW7 2AZ, UK 
}

\email {a.laptev@imperial.ac.uk}


\begin{abstract}
In this paper we establish a Hardy inequality for Laplace operators with Robin boundary conditions. For convex domains, in particular, we show explicitly how the corresponding Hardy weight depends on the coefficient of the Robin boundary conditions. We also study several extensions to non-convex and unbounded domains.  
\end{abstract}

\maketitle

{\bf  AMS 2000 Mathematics Subject Classification:} 47F05, 39B72\\

{\bf  Keywords:}
Robin Laplacian, Hardy inequality \\


\section{Introduction}

The classical Hardy inequality states that if $n\ge3$ then for any
function $u$ such that  $\nabla u\in L^2(\Bbb R^n)$ it holds
\begin{equation}\label{HClass}
\int_{\Bbb R^n} |\nabla u(x)|^2\, dx \ge \Big(\frac{n-2}{2}\Big)^2 \, \int_{\Bbb R^n} \frac{|u(x)|^2}{|x|^2}\, dx.
\end{equation}
It is well known that the constant $(n-2)^2/4$ in \eqref{HClass} is sharp but not achieved. 
The literature concerning different versions of Hardy's inequalities and their applications is extensive and we are not able to cover it in this paper. We just mention the classical paper of M.Sh.Birman \cite{B} and the books of E.B.Davies \cite{D1,D3} and V.Maz'ya \cite{M}.

A version of Hardy inequalities was considered by E.B. Davies (see for example \cite{D2} or \cite[Sect.5.3.]{D3}) who, in particular, proved that for convex domains $\Omega\subset \mathbb R^n$, $n\ge 2$
\begin{equation}\label{HDavies}
\int_\Omega |\nabla u(x)|^2 \, dx \ge \frac14\, \int_\Omega \frac{|u(x)|^2}{ \delta^2(x)} \, dx,
\qquad u\in H^1_0(\Omega),
\end{equation}
where $\delta$ is  the distance function  to the boundary $\partial\Omega$,
\begin{equation}\label{min}
\delta(x) = {\rm dist}\, (x,\partial\Omega) = \min_{y\in\partial\Omega} |x-y|.
\end{equation}
The $L^p$ version of inequality \eqref{HDavies}, for $p>1$, with the sharp constant was proven in \cite{MS} for $n=2$ and later in \cite{MMP} for any $n$. 

In the paper \cite{BM} H.Brezis and M. Marcus showed that if $\Omega\subset\R^n$, $n\ge 2$ is convex, then the inequality  \eqref{HDavies} could be improved to include the $L^2$-norm:
\begin{equation}\label{hardyBM}
\int_{\Omega} |\nabla u|^2 \,dx \ge \frac{1}{4}\int_{\Omega} \frac{|u|^2}{\delta^2(x)}\, dx
+ C(\Omega) \,\int_{\Omega} |u|^2 dx,
\end{equation}
where $C(\Omega) = c(\rm{diam} \, \Omega)^{-2}$ with some $c>0$.
They also conjectured that $C(\Omega)$ should depend on
the Lebesgue measure of $\Omega$. This conjecture was justified in \cite{HHL} where it was proved that 
in \eqref{hardyBM} $C(\Omega)$ could be chosen as $c|\Omega|^{-2/n}$ with some $c>0$ independent of $\Omega$. This result was later 
generalised to $L^p$-type inequalities  in \cite{T}.

\noindent
If $\Omega$ is not convex then generally speaking the constant in front of the first integral in the right hand side of \eqref{hardyBM} could be less then $1/4$. In 1986 A. Ancona \cite{A} showed using the Koebe one-quarter Theorem, that for a simply-connected planar domain the constant in the Hardy inequality with the distance to the boundary is greater than or equal to $1/16$. In \cite{LS} the authors have considered classes of domains for which there is a stronger version of the Koebe Theorem and which in turns implies better estimates for the constant appearing in the Hardy inequality \eqref{HDavies}.

In \cite{FMT}  S. Filippas,  V.G. Maz'ya and A. Tertikas (see also F.G. Avkhadiev \cite{A}) have obtained
that for convex domains $\Omega$ the constant $C(\Omega)$ in \eqref{hardyBM} could be expressed in terms of the inradius of $\Omega$. Namely if
$$
R_{in} := \sup\{\delta(x) \,:\, x\in\Omega\},
$$
then $C(\Omega) = c^2_0\, R_{in}^{-2}$ with some $c_0>0$.
Recently F.G. Avkhadiev and K.-J. Wirths \cite{AW} have shown that the best possible constant $c_0$  equals 
 the first positive zero  of the function 
$$
J_0(t) - 2t J_1(t) = J_0(t) + 2tJ_0'(t)
$$
where $J_0$ and $J_1$ are the Bessel functions of order $0$ and $1$ respectively.
Note, however, that the results obtained in the papers \cite{FMT}, \cite{A} and  \cite{AW} do not cover the case of non-convex domains whereas it has been showed in \cite{HHL} that the remainder term $c|\Omega|^{-2/d}$ survives even for non-convex domains. 

\medskip

In this paper we are not going to consider the classical Dirichlet-Laplacian, but the  so-called Robin-Laplacian generated by the quadratic form
\begin{equation} \label{form}
\Q_\sigma[u]= \int_\Omega |\nabla u(x)|^2 \,dx + \int_{\partial\Omega} \sigma(y)\, |u(y)|^2\, d\nu(y),
\end{equation}
where $d\nu$ denotes the surface measure on $\partial\Omega$ and $\sigma$ is a measurable function which defines the boundary conditions. If $\sigma\in L^\infty(\partial\Omega)$ is non-negative and such that $\sigma>0$ on a part of the boundary of non-zero surface measure and if $\Omega$ is bounded and regular enough, see section \ref{sect-prelim}, then $\Q_\sigma[\cdot]$ is positive definite on $H^1(\Omega)$ and therefore must satisfy some Hardy type inequality with a positive integral weight. 

Our aim is to find out how such a Hardy inequality depends on the function $\sigma$ and on the geometry of $\Omega$. We will deal with several types of domains in $\R^n$. For convex domains we establish a Hardy inequality with an explicit expression for the associated integral weight, Theorem \ref{main}. A generalisation to non-convex domains is discussed in section \ref{sect-general}, see Theorem \ref{pseudo-ndim} and Corolary \ref{cor-general}. In the closing section \ref{sect-unb} we treat an example of an unbounded non-convex domain.


\section{Preliminaries} \label{sect-prelim}

\noindent
Let $\Omega\subset\R^n$ be an open domain. We will need the following hypothesis.  

\begin{assumption} \label{ass-basic}
$\Omega$ satisfies the strong local Lipschitz condition. This means that each point $y\in\partial\Omega$ has a neighbourhood $U_y$ such that $U_y\cap \partial\Omega$ is the graph of a Lipschitz continuous function with the Lipschitz constant independent of $y$, see \cite[Chap.4]{Ad}.   
\end{assumption}

\noindent
If in addition to Assumption \ref{ass-basic} we suppose that $\Omega$ is bounded and that 
$\sigma\in L^\infty(\partial\Omega)$, then in view of the trace inequality, see e.g. \cite[Sect.7.5]{Ad}, 
it follows that
\begin{equation} \label{trace}
\Q_\sigma[u] \,  \leq c\, \|u\|_{H^1(\Omega)} 
\end{equation}
for some $c$ and all $u\in H^1(\Omega)$. Hence the quadratic form $\Q_\sigma[u]$
defined on $H^1(\Omega)$ is closed. The unique self-adjoint operator generated by $\Q_\sigma$ is then the Robin-Laplacian which formally satisfies the boundary conditions
$$
 \frac{\partial u}{\partial \eta_y}(y) + \sigma(y) u(y) =0 , \qquad y\in\partial\Omega,
$$
where $\eta_y$ denotes the unit outer normal vector at $y\in\partial\Omega$. 

Finally, let us denote by $S\subset\Omega$ the subset of points in $\Omega$ for
which there exist at least two points $y_1, y_2 \in \partial\Omega$ where
the minimum in \eqref{min} is achieved. Usually this set is called the
{\em singular} set of $\Omega$ and it is known that its Lebesgue
measure is zero (see for example \cite{LN}). We introduce the 
projection $p:\, \Omega\setminus S \to
\partial\Omega $ by
\begin{equation} \label{projection}
p(x) :=  y\in \partial\Omega\ : \, \delta(x) = |x-y|, \quad x\in \Omega\setminus S.
\end{equation} 

\smallskip


\section{Convex domains} \label{sect-convex}

\subsection{Bounded convex domains}

\noindent
Our main result is the following 

\begin{theorem}\label{main} 
Let $\Omega\subset\R^n$ be open bounded and convex. Then for any $0\leq \sigma \in L^\infty(\partial\Omega)$ and all $u\in H^1(\Omega)$ it holds
\begin{align} \label{hardy-main}
& \int_{\Omega} |\nabla u(x)|^2 \, dx +\! \int_{\partial\Omega}\! \sigma(y) \, |u(y)|^2\, d\nu(y) 
\,  \geq \,  \frac14 \int_{\Omega} \Big
(\delta(x)+\frac{1}{2\sigma(p(x))}\Big)^{-2} |u(x)|^2\, dx \,  +  \\
& \qquad +\frac14 \int_{\Omega}  \Big (R_{in}+\frac{1}{2\sigma(p(x))}\Big)^{-2} |u(x)|^2\, dx
+ \frac 12 \int_{\partial\Omega}\Big(R_{in} +\frac{1}{2\sigma(y)}\Big)^{-1} \, |u(y)|^2 \, d\nu(y). \nonumber
\end{align}
\end{theorem}

\medskip

\noindent Proof of Theorem \ref{main} is given after Lemma \ref{lem1}.

\medskip

\begin{remark}
Note  that $p(x)$ is defined on $\Omega\setminus S$ and therefore
almost everywhere in $\Omega$.
\end{remark}

\begin{remark}
Assumption \ref{ass-basic} is satisfied for all bounded convex domains. 
\end{remark}

\begin{remark}
Let us discuss the optimality of the first term on the right hand side of inequality \eqref{hardy-main}. Assume that  
\begin{equation} \label{ass}
\int_{\Omega} |\nabla u|^2\, dx + \sigma\!\int_{\partial\Omega} u^2\,
d\nu  \, \geq \, c \int_\Omega\,
\frac{|u|^2}{(\delta+\frac{1}{2\sigma})^2}\, \, dx.
\end{equation}
holds true for a constant $\sigma>0$ and some $c$. Let $\Omega\subset\R^n$
be a ball of radius $R$ centered in the origin and consider the
function
$$
u(x) = \Big(R+\frac{1}{2\sigma}\, -|x|\Big)^{1/2}\, , \qquad x\in\Omega.
$$
By inserting $u$ into \eqref{ass} and letting $R\to\infty$ we get
\begin{align*}
\left(1-4c\right)\, R^{n-1}\, \log(1+2\sigma R) + O(R^{n-1}) \geq 0,
\qquad R\to\infty.
\end{align*}
This shows that $c\leq 1/4$ and hence the constant $1/4$ on the firs line of \eqref{hardy-main} is sharp. Another way to see this is to look at the limit $\sigma \to\infty$, cf. Corollary \ref{Dirichlet} below, in which we recover the classical Hardy inequality where the constant $1/4$ is optimal. 

However, it turns out that for a fixed convex domain $\Omega$ and certain values of $\sigma$, with $\sigma$ being equal to a constant,  it is possible to obtain a better weight than the one given by the first term in \eqref{hardy-main}, 
see Remark \ref{added} for more detail.
\end{remark}

\medskip

\noindent
Theorem \ref{main} can be applied also in situations with Dirichlet boundary condition on some part of the boundary. In order to formulate the corresponding statement, we introduce the space 
$$
C^1_{0,\Gamma}(\overline{\Omega}) := \big\{ u\in C^1(\overline{\Omega})\, :\, u |_{\Gamma} \, =0 , \ \Gamma\subset\partial\Omega\big\}.
$$ 
We then have 

\begin{corollary} \label{partial-Dirichlet}
Let $\Omega$ be as in Theorem \ref{main} and let $\Gamma\subset\partial\Omega$ be closed. Then for all $u\in C^1_{0,\Gamma}(\overline{\Omega})$ and any $\tilde\sigma\in L^\infty(\partial\Omega\setminus\Gamma)$ inequality \eqref{hardy-main} holds true with 
$$
\sigma(y) = \left\{
\begin{array}{l@{\qquad}l}
 \tilde\sigma(y)    &      {\rm if \ } \  y\in\  \partial\Omega\setminus\Gamma ,  \\
 +\infty     &   {\rm if \ } \  y\in\Gamma .
\end{array}
\right. 
$$
\end{corollary}

\begin{proof}
Let $u\in C^1_{0,\Gamma}(\overline{\Omega})$ and define the sequence of functions $\sigma_n: \partial\Omega \to \R_+$ by
$$
\sigma_n(y) = \left\{
\begin{array}{l@{\qquad}l}
 \tilde\sigma(y)    &       {\rm if \ } \   y\in\  \partial\Omega\setminus\Gamma ,  \\
n     &    {\rm if \ } \  y\in\Gamma 
\end{array}
\right. \qquad n\in\N.
$$
Then $\sigma_n\in L^\infty(\partial\Omega)$ for each $n\in\N$ and therefore we can apply Theorem \ref{main}. The statement then follows from the monotone convergence theorem by letting $n \to\infty$.
\end{proof}

\noindent When $\Gamma=\partial\Omega$, then the above Corollary yields an improvement of the classical Hardy inequality \eqref{HDavies} for Dirichlet Laplacians, see \cite{A1,A2,AW,BM,FMT,HHL} for more results in this direction. 

\begin{corollary}\label{Dirichlet}
Let $\Omega$ be as in Theorem \ref{main}. Then for all $u\in H^1_0(\Omega)$ it holds
\begin{align} \label{hardy-dirichlet}
& \int_{\Omega} |\nabla u(x)|^2 \, dx\,  \geq \,  \frac14 \int_{\Omega} 
\frac{|u(x)|^2}{\delta^2(x)} \, \, dx  + \frac{1}{4\, R_{in}^2}  \int_{\Omega} |u(x)|^2  \, dx.
\end{align}
\end{corollary}

\smallskip

\begin{remark}
In \cite{BM} and \cite{HHL} the constant in front of the second term on the right hand side is expressed in terms of the diameter and volume of $\Omega$ respectively. A sharp constant involving the inradius $R_{in}$ is due to \cite{AW}.
\end{remark}

\smallskip


\noindent
The key idea of the proof of Theorem \ref{main} is to establish inequality \eqref{hardy-main} first for convex polytops, and then to approximate the domain $\Omega$ by a sequence of such polytops. We start with a one-dimensional result. 

\begin{lemma}\label{lem1}
Let $b>0$ and assume that $u$ belongs to $AC[0,b]$, the space of absolutely continuous functions on $[0,b]$.  Then for any $\sigma\geq 0$ we have
\begin{align} \label{cor2-eq}
\int_0^b |u'(t)|^2 \, dt +  \sigma\, |u(0)|^2 & \, \geq\,  \frac14\,
\int_0^b \Big(\Big (t+\frac{1}{2\sigma}\Big)^{-2} + \Big
(b+\frac{1}{2\sigma}\Big)^{-2}\Big)\, |u(t)|^2\, dt \nonumber \\
&\qquad \qquad + \frac 12\, \Big(b +\frac{1}{2\sigma}\Big)^{-1} \,
|u(0)|^2.
\end{align}
\end{lemma}

\begin{proof}  The inequality is trivial for $\sigma=0$. Hence we may assume that $\sigma>0$. 
Let $f\in C^1[0,b]$. Integration by parts and Cauchy-Schwartz
inequality show that
\begin{align*}
\Bigl( \int_0^b (f(t)-f(b))' \,|u|^2\, dt - & (f(b)-f(0))
|u(0)|^2\Bigr)^2 = \Bigl(\int_0^b (f(b)-f(t))\,(u'\bar u + u\, \bar
u')\, dt\Bigr)^2
\\
& \qquad \qquad \leq
4\Bigl(\int_0^b|u'|^2\,dt\Bigr)\Bigl(\int_0^b(f(b)-f(t))^2\,|u|^2\,dt\Bigr).
\end{align*}
This together with the inequality
$$
\frac{A^2}{B} \ge 2A - B, \qquad B>0,
$$
gives
\begin{equation} \label{inter}
\int_0^b |u'(t)|^2\, dt  +\frac{f(b) - f(0)}{2}\,  |u(0)|^2 \geq
\frac{1}{2}\, \int_0^b f^{'}(t) |u|^2\, dt - \frac 14\, \int_0^b
(f(b)-f(t))^2 |u|^2\, dt.
\end{equation}
By inserting
$$
f(t) = -\Big(t +\frac{1}{2\sigma}\Big)^{-1}
$$
into \eqref{inter} and using the fact that
$$
\Big (t+\frac{1}{2\sigma}\Big)^{-1}\Big
(b+\frac{1}{2\sigma}\Big)^{-1} \geq \Big
(b+\frac{1}{2\sigma}\Big)^{-2}, \qquad 0\leq t \leq b,
$$
we arrive at \eqref{cor2-eq}.
\end{proof}

\begin{remark} \label{added}
Although the right hand side of \eqref{cor2-eq} is optimal in the limit $\sigma\to\infty$, the weight in the first term might be improved for $\sigma$ large enough, depending on $b$, but finite. This follows from the proof of \cite[Lemma A.1]{BM} applied to a test function $v(t)=(t+2/\sigma)^{-1/2} u(t)$ with $u\in C^1[0,b]$. 
\end{remark}


\begin{proof}[{\bf Proof of Theorem \ref{main}}] 
We start by proving the statement for $u\in C^1(\overline{\Omega})$. Assume first that $\sigma$ is continuous. Since $\partial\Omega$ is closed in $\R^n$, by Tietze's extension Theorem there exists a continuous function $\Sigma: \, \Omega\to \R$ whose restriction to $\partial\Omega$ coincides with $\sigma$.

Let $Q\subset\Omega$ be an open convex polytop in $\Bbb R^n$ with $N$ sides
$\Gamma_j$, $j=1,\dots N$. Clearly we have $u\in C^1(\bar Q)$. Denote by $n_j$ the unit inner normal vector to $\Gamma_j$.
For each side $\Gamma_j$ we consider the domain $P_j$ attached to $\Gamma_j$ by including all the points from $Q$ for which the distance to the boundary $\partial Q$ is achieved at a point belonging to $\Gamma_j$. More precisely, 
$$
P_j =\big\{x\in Q: \,  \exists \ y\in \Gamma_j:
{\rm dist} (x, \partial Q) = |x-y|\big\}.
$$
Then $Q = \cup_j \overline{P}_j$ and the singular set of $Q$ is given by $S= \cup_{j=1}^N (\partial P_j\setminus \Gamma_j)$. The inradius of $\Omega$ obviously satisfies 
\begin{equation} \label{smallerR}
R_{in} \geq \max_j \, \max_{x\in P_j}\,  {\rm dist}\, (x,\Gamma_j).
\end{equation}
Moreover, for each $y\in\Gamma_j$ there is a unique point $x_y\in S$ and $t_y$ such that
\begin{equation} \label{ty}
x_y =   y+ t_y\, n_j .
\end{equation}
Denote by $p_Q$ the projection on $\partial Q$ defined in the same way as $p$ in \eqref{projection} with $\Omega$ replaced by $Q$. We apply Lemma \ref{lem1} along the normal vector $n_j$ and, taking into account \eqref{smallerR},  we obtain
\begin{multline*}
\int_0^{t_y}  |\partial_t u(y +t\, n_j)|^2 \, dt +  \Sigma_Q(y) \, |u(y)|^2\\
\ge
\frac14\, \int_0^{t_y} \Big(\Big (t+\frac{1}{2\Sigma_Q(y)}\Big)^{-2}
+  \Big (R_{in}+\frac{1}{2\Sigma_Q(y)}\Big)^{-2}\Big)\, |u(y +t\, n_j)|^2\, dt\\
+ \frac 12\,  \Big(R_{in} +\frac{1}{2\Sigma_Q(y) }\Big)^{-1} \,
|u(y)|^2,
\end{multline*}
where $\Sigma_Q$ is the restriction of $\Sigma$ on $\partial Q$.
Next we note that 
$$
t = \delta(y+t\, n_j) \qquad   \forall\, y\in \Gamma_j, \ \ \forall\ t\in (0,t_y). 
$$
Hence, integrating over variables orthogonal to $n_j$ and using the  invariance of the Laplacian with respect to rotations we arrive at 
\begin{align}
& \int_{P_j} |\nabla u(x)|^2 \, dx + \int_{\Gamma_j} \Sigma_Q(y) \, |u(y)|^2\, dy  \geq
\frac14\, \int_{P_j} \Big(\Big (\delta(x)+\frac{1}{2\, \Sigma_Q(p_Q(x))}\Big)^{-2}
\label{eq:polygons} \\ 
& \qquad +  \Big (R_{in}+\frac{1}{2\, \Sigma_Q(p_Q(x))}\Big)^{-2}\Big)\, |u(x)|^2\, dx  
+ \frac 12\,  \int_{\Gamma_j}\Big(R_{in}
+\frac{1}{2\Sigma_Q(y)}\Big)^{-1} \, |u(y)|^2 \, dy,  \nonumber 
\end{align}
where $dy$ denotes the $(n-1)-$dimensional Lebesgue measure on $\Gamma_j$. 
Summation of \eqref{eq:polygons} over $j$ gives us inequality \eqref{hardy-main} for any convex polytop $Q\subset\Omega$ with $\sigma$ replaced by $\Sigma_Q$ and $p$ replaced by $p_Q$.  

Since $\Omega$ is convex, there exists a sequence of $n-$dimensional convex polytops $Q_m\subset\Omega,\, m\in\N$, which approximates $\Omega$. More precisely, for every $\varepsilon$ there exists an $m_\eps$ such that the Hausdorf distance between $\Omega$ and $Q_{m_\eps}$ satisfies $d_H(\Omega, Q_{m_\eps}) < \varepsilon$, see \cite[\S 9]{had}. Since $Q_m$ approximates $\Omega$ also in the surface measure, \cite[\S 14]{had},  from the continuity of $u$ and $\Sigma$ it follows that 
$$
\Sigma_{Q_m}(p_{Q_m}(x))\  \to \ \sigma(p(x)) \qquad \text{a. e.} \quad  x\in\Omega, 
$$
and
$$  
\int_{\partial Q_m}\! \Sigma_{Q_m}(y) \, |u(y)|^2\, d y\ \to \  \int_{\partial\Omega}\! \sigma(y) \, |u(y)|^2\, d\nu(y)
$$
as $m\to\infty$. Hence, by the dominated convergence we can pass to the limit to get 
\begin{align}
& \int_{\Omega} |\nabla u(x)|^2 \, dx + \int_{\partial\Omega} \sigma(y)\, |u(y)|^2\, d\nu(y)  \geq
\frac14\, \int_{\Omega} \Big(\Big (\delta(x)+\frac{1}{2\, \sigma(p(x))} \Big)^{-2}
 \label{eq:u-cont} \\ 
& \qquad +  \Big (R_{in}+\frac{1}{2\, \sigma(p(x))}\Big)^{-2}\Big)\, |u(x)|^2\, dx  
+ \frac 12\,  \int_{\partial\Omega}\Big(R_{in} +\frac{1}{2\sigma(y)}\Big)^{-1} \, |u(y)|^2 \, d\nu(y). \nonumber 
\end{align}
for all $u\in C^1(\overline{\Omega})$ and any $\sigma$ continuous. 

Next we note that if $\sigma\in L^\infty(\partial\Omega)$, then in view of the regularity of $\partial\Omega$ there exists a sequence of continuous functions $\sigma_k$ on $\partial\Omega$ which converges to $\sigma$ in $L^1(\partial\Omega)$ as $k\to\infty$. Then $\sigma_k$ admits a subsequence, which we still denote by $\sigma_k$, such that $\sigma_k \to \sigma$ almost everywhere on $\partial\Omega$. Hence
$\sigma_k(p(x)) \to \sigma(p(x))$  for almost every $x\in\Omega$.
From inequality \eqref{eq:u-cont} it follows that \eqref{hardy-main} holds for  all $\sigma_k$.  
Since $u |_{\partial\Omega} \in L^\infty(\partial\Omega)$, we can pass to the limit as $k\to\infty$ and using  the dominated convergence we obtain inequality \eqref{hardy-main} for any $\sigma\in L^\infty(\partial\Omega)$ and all $u\in C^1(\overline{\Omega})$. 

Finally, if $u\in H^1(\Omega)$, then by density there exists a sequence $u_j \in C^1(\overline{\Omega})$ such that $\|u-u_j\|_{H^1(\Omega)} \to 0$ as $j\to \infty$. By the regularity of $\Omega$ it follows that $H^1(\Omega) \hookrightarrow L^2(\partial\Omega)$,  \cite[Sect.7.5]{Ad}. Hence, after applying inequality \eqref{eq:u-cont} to $u_j$ and letting $j\to \infty$ we conclude that \eqref{hardy-main} holds for all $u\in H^1(\Omega)$.  
\end{proof}

\subsection{Unbounded convex domains} For unbounded domains we need to impose some decay conditions on the test functions. Let $\rho>0 $ and define 
\begin{equation} \label{space-1}
\dot{C}^1(B_\rho) = \big\{ u\in C^1(\R^n)\, : \, u(x)= 0 \quad  \forall\, x \, :\, |x| \geq \rho \big\}, 
\end{equation}
where $B_\rho =  \{ x\in\R^n\, : \, |x| < \rho\}$. We have 

\begin{theorem}
Let $\Omega$ be open and convex and let $0\leq \sigma \in L^\infty(\partial\Omega)$. Let $\rho >0$. Then inequality \eqref{hardy-main} holds for all $u\in \dot{C}^1(B_\rho)$. 
\end{theorem}

\begin{proof} 
Let $u\in \dot{C}^1(B_\rho)$.  Define $\Omega_\rho = \Omega \cap B_\rho$. Then $\Omega_\rho$ is convex and bounded. Now define $\hat\sigma:\partial\Omega_\rho \to \R_+$ by 
$$
\hat \sigma(y) = \left\{
\begin{array}{l@{\qquad}l}
 \sigma(y)    &        {\rm if \ } \  y\in\  \partial\Omega\setminus\partial B_\rho ,  \\
 +\infty     & \text{elsewhere}.
\end{array}
\right. 
$$
Let $p_\rho$ be the projection on $\partial\Omega_\rho$ defined in the same way as $p$ in \eqref{projection} with $\Omega$ replaced by $\Omega_\rho$. Accordingly, let $\delta_\rho(x) = $dist$(x,\partial\Omega_\rho)$. Then
$$
\hat\sigma(p_\rho(x)) \geq \sigma(p(x)), \quad \delta_\rho(x) \leq \delta(x) \qquad \text{for\, a. e.}\quad  x\in \Omega_\rho.
$$
Moreover, the inradius of $\Omega_\rho$ is less or equal to the inradius $R_{in}$ of $\Omega$.  Since $u |_{\Omega_\rho} \in H^1(\Omega_\rho)$ we can apply Theorem \ref{main} to $\Omega_\rho$ with $\hat\sigma$ defined as above. This gives the result. 
\end{proof}


\section{General bounded domains} \label{sect-general}

In this section we will establish a version of the Hardy inequality on general open domains $\Omega\subset\R^n$ satisfying Assumption \ref{ass-basic}. We will follow the approach of \cite[Sec.1.5]{D1}, see also \cite{D2}. As in the case of convex domains we start with an auxiliary one-dimensional result.  

\begin{lemma}\label{lem2}
Let $u\in AC[0,b]$. Then for any $\sigma_1\geq 0$ and $\sigma_2 \geq 0$ we have
\begin{align} \label{lem2-eq}
\int_0^b |u'(t)|^2 \, dt +  \sigma_1\, |u(0)|^2 + \sigma_2\, |u(b)|^2 & \, \geq\,  \frac14\,
\int_0^{b/2} \Big (t+\frac{1}{2\sigma_1}\Big)^{-2}\, |u(t)|^2\, dt  \\
& \quad + \frac 14\, 
\int_{b/2}^b \Big (b-t+\frac{1}{2\sigma_2}\Big)^{-2}\, |u(t)|^2\, dt. \nonumber
\end{align}
\end{lemma}

\begin{proof}
This follows  immediately from inequality \eqref{cor2-eq}.  
\end{proof}

\noindent Given  $x\in\Omega$ and some $e\in\R^n$ with $\|e\|_n=1$, we introduce
\begin{equation} \label{min2}
d_e(x)= \min \big\{\, |s|\, : \,  s\in\R\, :\, x+s\, e
\notin\Omega\big\}
\end{equation}
Hence, $d_e(x)$ is the distance from $x$ to the boundary of $\Omega$ in the direction of the vector $e$.  Denote by $m(e,x)$ the set on which the minimum in \eqref{min2} is achieved. Clearly, $m(e,x)$ contains either one or two elements. Let us define 
\begin{equation}  
\sigma_e(x) := \max_{s\in m(e,x)} \sigma(x+s e),  \qquad x \in\Omega.
\end{equation}
Finally, let d$\sph$ be the normalized surface measure on the unit sphere in $\R^n$ and define
\begin{equation} \label{mu}
\mu_\sigma(x) := \int_{e:\, \|e\|=1}
\Big(d_e(x)+\frac{1}{2\, \sigma_e(x)}\Big)^{-2}\, {\rm d}\sph(e).
\end{equation}
We have 

\begin{theorem} \label{pseudo-ndim}
Let $\Omega \subset \R^n$ be a bounded domain satisfying Assumption \ref{ass-basic}.  
Then for all $u\in H^1(\Omega)$ and any $\sigma\in L^\infty(\partial\Omega)$ it holds
\begin{equation} \label{hardy:pseudo:2dim}
\int_{\Omega} |\nabla u(x)|^2\, dx + \int_{\partial\Omega}
\sigma(y)\, |u(y)|^2\, d\nu(y) \, \geq \, \frac 14 \int_\Omega\,
|u(x)|^2\, \mu_\sigma(x)\, \, dx.
\end{equation}
\end{theorem}

\begin{proof}
By Assumption \ref{ass-basic} and \cite[Thm.3.22]{Ad} 
it suffices to prove inequality \eqref{hardy:pseudo:2dim}  for $u\in C^1(\overline{\Omega})$. We denote by $\partial_e$ the partial differentiation in direction $e$. From Lemma \ref{lem2} we find out that 
\begin{align} \label{check}
\frac 14 
\int_\Omega\,  \Big(d_e(x)+\frac{1}{2\, \sigma_e(x)}\Big)^{-2}\, |u(x)|^2\, dx 
&
\leq \int_{\Omega} |\partial_e u(x)|^2\, dx + \int_{\partial\Omega} \sigma(y)\, |u(y)|^2\, d\nu(y)
\end{align}
holds for all $e$ with $\|e\|=1$. The result then follows by integrating the last inequality with respect to $e$ over the unit sphere.
\end{proof}

\vspace{0.2cm}

\begin{corollary} \label{cor-general}
Let $\Omega$ satisfy conditions of Theorem \ref{pseudo-ndim}. Assume that $\sigma(y)=\sigma>0$ is constant. Then there exists a constant $K=K(\Omega,n)$, independent of $\sigma$, such that for all $u\in H^1(\Omega)$ it holds
\begin{equation} \label{eq:hardy:2dim}
\int_{\Omega} |\nabla u(x)|^2\, dx + \sigma\! \int_{\partial\Omega}
|u(y)|^2\, d\nu(y) \, \geq \,   K \!
\int_\Omega\, \frac{|u(x)|^2}{(\delta(x)+\frac{1}{4\sigma})^2}\, \, dx.
\end{equation}
\end{corollary}

\begin{proof} Denote by $\omega_n$ the surface area of the unit ball in $\R^n$.
Let $x\in\Omega$ and let $a\in\partial\Omega$ be such that
$r:=\delta(x) =|x-a|$. Define
$$
\Lambda_x = \big\{\, e\in\R^n\, :\, \|e\|=1\, \wedge \, \exists\,  s>0\, \,
\text{such\, that}\, \, x+ s\, e \notin \Omega \, \,
\text{and}\, \, |x+s\, e-a| <r \big\}.
$$
By the triangle inequality we have $s < 2r$. It follows that
the region $\{y\notin\Omega\, :\, |y-a| <r\}$ is contained in the
$n-$dimensional cone of radius $2r$ with the opening angle
$\sph(\Lambda_x)\, \omega_n$ centered in $x$. Hence
\begin{equation} \label{cone}
c_n\, \sph(\Lambda_x)\, r^n \geq \text{Vol}\big (\{y\notin\Omega\, :\, |y-a|
<r\}\big) ,
\end{equation}
for some constant $c_n$ which depends only on $n$. 
On the other hand, from Assumption \ref{ass-basic} it follows that there exits $\alpha>0$ such that 
\begin{equation*}
\text{Vol}\big (\{y\notin\Omega\, :\, |y-a| <r\}\big) \geq  \alpha\, r^n \qquad \forall\, a\in\partial\Omega, \, \, \forall\, r>0.
\end{equation*}
This together with \eqref{cone} implies that $\sph(\Lambda_x) \geq c_n^{-1}\, \alpha$ for every $x\in\Omega$.
We thus get
\begin{align*}
\mu_\sigma(x) & = \int_{e:\, \|e\|=1}
\Big(d_e(x)+\frac{1}{2 \sigma}\Big)^{-2}\, {\rm d}\sph(e) \geq 
 \int_{\Lambda_x}
\Big(d_e(x)+\frac{1}{2 \sigma}\Big)^{-2}\, {\rm d}\sph(e) \\
& \geq  \Big(2\, \delta(x)+\frac{1}{2 \sigma}\Big)^{-2}\,  \sph(\Lambda_x) \geq
\Big(2\, \delta(x)+\frac{1}{2 \sigma}\Big)^{-2}\,   \frac{\alpha}{c_n},
\end{align*}
since for every $e\in\Lambda_x$ we have $d_e(x) \leq s
\leq 2\, \delta(x)$. Inequality \eqref{eq:hardy:2dim} now follows
from Theorem \ref{pseudo-ndim}.
\end{proof}

\subsection{The case of sign changing $\sigma$} Note that the assumption $\sigma\geq 0$ is crucial for the results given in Theorems \ref{main} and \ref{pseudo-ndim}. When $\sigma$ is negative on some part of $\partial\Omega$, then a simple test function argument shows that the resulting Robin-Laplacian may have a negative eigenvalue even if 
$$
\int_{\partial\Omega}\! \sigma\, d\nu >0,
$$  
provided $\sigma$ is chosen in a suitable way. 
This tells us that if $\sigma$ changes sign, then no Hardy inequality with positive integral weight can hold unless some further restrictions are made. 

In order to give an example of a Hardy inequality with a sign changing weight, we consider a class of domains characterized as follows. Suppose that $f :\R^{n-1} \to \R$ is continuous and that there exists an open set $A\subset \R^{n-1}$ such that $f >0$ in $A$ and $f= 0$ on $\partial A$. We then define
\begin{equation}
\Omega = \{\, x:=(x',t) \in \R^n\, :\, x' \in A\, , \, 0 < t < f(x') \}.
\end{equation}
Let $\sigma\in L^\infty(\partial\Omega)$ be such that $\sigma = 0$ on $\partial\Omega \setminus A$ and denote by $A_{\pm} \subset A$ the sets on which $\sigma$ is positive respectively negative. Let $\chi_{A_\pm}$ be the related characteristic functions. We have 

\begin{proposition} \label{sigma-neg-1}
Let $\Omega$ and $\sigma$ be given as above. Then for any $u\in H^1(\Omega)$ it holds 
\begin{align}
 \int_{\Omega} |\nabla u(x)|^2 \, dx +\! \int_{\partial\Omega}\! \sigma(y) \, |u(y)|^2\, d\nu(y) & \geq 
 \int_{\Omega} \rho(x)\, |u(x)|^2\, dx  \, + \label{eq:sigma-neg} \\
& \quad + \frac 12 \int_{A_+} \Big(f(x') +\frac{1}{2\sigma(x')}\Big)^{-1}\, |u(x',0)|^2\, dx',  \nonumber
\end{align}
where 
$$
\rho(x) =\rho(x',t) = \frac 12 \big(f(x') +\frac{1}{2\sigma(x')}\big)^{-2} \chi_{A_+} (x') + \mu(x')\, \chi_{A_-}(x'),
$$ 
and $-\mu(x')$ is the first positive solution to the implicit equation  
\begin{equation} \label{mu-ev}
-\sigma(x') = \sqrt{- \mu(x')} \, \, \tanh \big( f(x') \sqrt{- \mu(x')} \big), \qquad   x' \in A_-. 
\end{equation}
\end{proposition}

\begin{proof}
Let $x'\in A_-$. A straightforward calculation shows that $\mu(x')$ defined by \eqref{mu-ev} is the lowest eigenvalues of the Laplace operator on interval $(0, f(x'))$ with Neumann boundary condition at $f(x')$ and Robin boundary condition with coefficient $\sigma(x')$ at $0$. Therefore we have 
$$
\int_0^{f(x')} |\partial_t u(x',t)|^2\, dt+\sigma(x')\, |u(x',0)|^2 \geq \mu(x')  \int_0^{f(x')}  |u(x',t)|^2\, dt \qquad \forall\ x'\in A_-.
$$
On the other hand, from Lemma \ref{lem1} we easily deduce that 
\begin{align}
\int_0^{f(x')} |\partial_t u(x',t)|^2\, dt +\sigma(x')\, |u(x',0)|^2 & \geq  \frac 12 \big(f(x') +\frac{1}{2\sigma(x')}\big)^{-2} \int_0^{f(x')}  |u(x',t)|^2\, dt \nonumber \\
&  + \frac 12 \big(f(x') +\frac{1}{2\sigma(x')}\big)^{-1}\, |u(x',0)|^2 \label{A+}
\end{align} 
for all $x'\in A_+$. 
\end{proof}

\section{Unbounded non-convex domains: an example} \label{sect-unb}

\noindent In this section we give an example of a Hardy inequality on a particular type of an unbounded non-convex domain. Namely, on the complement of a ball.   

\begin{theorem} \label{outsideball}
Let $B_R = \{ x\in\R^n\, : \, |x| < R\}$ and let $B^c_r=\R^n\setminus B_R$ be its complement. Then for any constant $\sigma \geq 0$ the inequality 
\begin{align} 
\int_{B^c_R} |\nabla u(x)|^2\, dx  + \sigma\int_{\partial B_R}
|u(y)|^2\, d\nu(y)  \, & \geq \, \frac 14 \int_{B^c_R} \Big (|x|-R+\frac{1}{2\sigma}\Big)^{-2} |u(x)|^2\, dx \, + 
& \label{hardy:outside} \\
& \quad +\frac{(n-1)(n-3)}{4} \int_{B^c_R} \frac{|u(x)|^2}{|x|^2}\, dx. \nonumber
\end{align}
holds true for all $u\in  H^1(B_R^c)$. 
\end{theorem}

\begin{proof}
Without loss of generality we may assume that $u$ is real-valued. Consider first the case $\sigma>0$. 
Let $\delta(x) = |x|-R$. We have $|\nabla \delta(x)|=1$ and $\Delta\delta(x)= \frac{n-1}{|x|}$. Integration by parts then gives
\begin{align}
 \int_{B_R^c} \frac{u(x)\, \nabla u(x) \cdot \nabla\delta(x)}{\delta(x)+\frac{1}{2\sigma}}\, dx& = 
\frac 12 \int_{B_R^c} \left( \frac{1}{(\delta(x)+\frac{1}{2\sigma})^2} -\frac{n-1}{|x|(\delta(x)+\frac{1}{2\sigma})}\right) \, |u(x)|^2\, dx \nonumber\\
& \quad - \sigma \int_{\partial B_R} |u(y)|^2\, d\nu(y),  \label{byparts-1}
\end{align}
and 
\begin{align} \label{final}
& \int_{B_R^c} \frac{u(x)\, \nabla u(x) \cdot x}{|x|^2} \, dx = -\frac{1}{2 R} \int_{\partial B_R} |u(y)|^2\, d\nu(y) -\frac{(n-2)}{2} \int_{B_R^c}  \frac{|u(x)|^2}{|x|^2}\, dx. 
\end{align}
\smallskip
Since $\nabla\delta(x) \cdot x= |x|$,  using \eqref{byparts-1}, \eqref{final} and a straightforward calculation we obtain 
\begin{align*}
& \int_{B_R^c} \Big |\, \nabla u(x) -\frac{\nabla\delta(x)}{2\delta(x)+\frac 1\sigma}\, u(x) +\frac{(n-1)\, u(x)\, x}{2 |x|^2}\, \Big|^2 \, dx \leq \int_{B_R^c} |\nabla u(x)|^2\, dx\, + \\
& \ \  \sigma\!\int_{\partial B_R}\! |u(y)|^2\, d\nu(y) 
 -\frac 14 \int_{B_R^c} \Big (\delta(x)+\frac{1}{2\sigma}\Big)^{-2} |u(x)|^2\, dx-\frac{(n-1)(n-3)}{4} \int_{B_R^c} \frac{|u(x)|^2}{|x|^2}\, dx.
\end{align*}
This proves the Theorem for $\sigma>0$. The case $\sigma=0$ then follows from \eqref{hardy:outside} by monotone convergence. 
\end{proof}


\section{Acknowledgements} 
A partial support from the MIUR-PRINÕ08 grant for the project  ''Trasporto ottimo di massa, disuguaglianze geometriche e funzionali e applicazioni''  (H.K.) is gratefully acknowledged. H.K. would like to thank the Department of Mathematics of the Imperial College London for the hospitality extended to him.


\end{document}